\documentclass[11pt]{article}

\usepackage{amssymb}
\usepackage{amsmath}
\usepackage{color}
\usepackage{theorem}

\usepackage{pgf}
\newdimen\acadpgfunit
\newdimen\acadpgflinewidth
\usepackage[all]{xy}
\newcommand{\XYMATRIX}{\xymatrix@M=6pt}
\newcommand{\aremb}{\ar@{^{(}->}}
\newcommand{\arembfrom}{\ar@{<-^{)}}}

\numberwithin{equation}{section}

\theorembodyfont{\slshape}

  \newtheorem{THM}{Theorem}[section]
  \newtheorem{LEM}[THM]{Lemma}

\theorembodyfont{\rmfamily}

  \newtheorem{DEF}[THM]{Definition}
  \newtheorem{EX}{Example}[section]

\newif\ifQEDsign
\newcommand{\QED}{\global\QEDsigntrue\hfill$\square$}

\newenvironment{proof}%
    {\par\noindent\textit{Proof.}\global\QEDsignfalse}%
    {\ifQEDsign\else\QED\fi\par\bigskip\par}

\newcommand{\place}[3]{\underset{\llap{\raisebox{-1.5mm}{\scriptsize$#1$#2 place}}\uparrow}{#3}}

\newcommand{\preclex}{\mathrel{\prec_{\mathit{lex}}}}
\newcommand{\precalex}{\mathrel{\prec_{\mathit{alex}}}}
\newcommand{\lex}{\mathrel{<_{\mathit{lex}}}}
\newcommand{\alex}{\mathrel{<_{\mathit{alex}}}}
\newcommand{\galex}{\mathrel{>_{\mathit{alex}}}}
\newcommand{\clex}{\mathrel{<_{\overline{\mathit{lex}}}}}
\newcommand{\clexprime}{\mathrel{<_{\overline{\mathit{lex}}}}}
\renewcommand{\le}{\leqslant}
\renewcommand{\ge}{\geqslant}

\newcommand{\0}{\varnothing}

\renewcommand{\sec}{\cap}
\renewcommand{\phi}{\varphi}
\renewcommand{\epsilon}{\varepsilon}
\newcommand{\UNION}{\bigcup}

\newcommand{\CC}{\mathbf{C}}
\newcommand{\DD}{\mathbf{D}}

\newcommand{\KK}{\mathbf{K}}

\newcommand{\NN}{\mathbb{N}}

\newcommand{\QQ}{\mathbb{Q}}
\newcommand{\RR}{\mathbb{R}}

\newcommand{\ZZ}{\mathbb{Z}}
\newcommand{\union}{\cup}
\newcommand{\restr}[2]{\hbox{$#1$}\hbox{$\upharpoonright$}_{#2}}
\newcommand{\Boxed}[1]{\mbox{$#1$}}

\newcommand{\id}{\mathrm{id}}

\newcommand{\Ob}{\mathrm{Ob}}

\newcommand{\spec}{\mathrm{spec}}

\newcommand{\calA}{\mathcal{A}}
\newcommand{\calB}{\mathcal{B}}
\newcommand{\calC}{\mathcal{C}}
\newcommand{\calD}{\mathcal{D}}
\newcommand{\calE}{\mathcal{E}}

\newcommand{\calG}{\mathcal{G}}

\newcommand{\calL}{\mathcal{L}}
\newcommand{\calM}{\mathcal{M}}

\newcommand{\calP}{\mathcal{P}}

\newcommand{\calS}{\mathcal{S}}

\newcommand{\calU}{\mathcal{U}}
\newcommand{\calV}{\mathcal{V}}
\newcommand{\calW}{\mathcal{W}}
\newcommand{\calX}{\mathcal{X}}
\newcommand{\calY}{\mathcal{Y}}

\newcommand{\catG}{\overrightarrow{\mathbf{G}}}
\newcommand{\catM}{\overrightarrow{\mathbf{M}}}
\newcommand{\catP}{\overrightarrow{\mathbf{P}}}
\newcommand{\catU}{\overrightarrow{\mathbf{U}}}
\newcommand{\catV}{\mathbf{V}}
\newcommand{\catGR}{\mathbf{GR}}

\title{Pre-adjunctions and the Ramsey property}
\author{%
  Dragan Ma\v sulovi\'c\\
  University of Novi Sad, Faculty of Sciences\\
  Department of Mathematics and Informatics\\
  Trg Dositeja Obradovi\'ca 3, 21000 Novi Sad, Serbia\\
  e-mail: dragan.masulovic@dmi.uns.ac.rs}

\begin{document}
\maketitle

\begin{abstract}
  Showing that the Ramsey property holds for a class of finite structures
  can be an extremely challenging task and a slew of sophisticated methods have been proposed in literature.

  In this paper we propose a new strategy to show that a class of structures has the Ramsey property.
  The strategy is based on a relatively simple result in category theory and
  consists of establishing a \emph{pre-adjunction} between the category
  of structures which is known to have the Ramsey property, and the category of structures we are interested in.
  
  We demonstrate the applicability of this strategy by providing short proofs of three important
  well known results: we show the Ramsey property for the category of all finite linearly ordered posets with embeddings,
  for the category of finite convexly ordered ultrametric spaces with embeddings, and for
  the category of all finite linearly ordered metric spaces (rational metric spaces) with embeddings.

  \bigskip


  \noindent \textbf{AMS Subj.\ Classification (2010):} 05C55, 18A99
\end{abstract}

\section{Introduction}

Generalizing the classical results of F.~P.~Ramsey from the late 1920's, the structural Ramsey theory originated at
the beginning of 1970’s in a series of papers (see \cite{N1995} for references).
We say that a class $\KK$ of finite structures has the \emph{Ramsey property} (RP) if the following holds:
for any number $k \ge 2$ of colors and all $\calA, \calB \in \KK$ such that $\calA$ embeds into $\calB$
there is a $\calC \in \KK$
such that no matter how we color the copies of $\calA$ in $\calC$ with $k$ colors, there is a \emph{monochromatic} copy
$\calB'$ of $\calB$ in $\calC$ (that is, all the copies of $\calA$ that fall within $\calB'$ are colored by the same color).

Showing that the Ramsey property holds for a class of finite structures $\KK$
can be an extremely challenging task and a slew of sophisticated methods have been proposed in literature.
These methods are usually constructive: given $\calA, \calB \in \KK$ and $k \ge 2$
they prove the Ramsey property directly by constructing a structure $\calC \in \KK$
with the desired properties.

It was Leeb who pointed out in 1970 \cite{leeb-cat} that the use of category theory can be quite helpful
both in the formulation and in the proofs of results pertaining to structural Ramsey theory.
Instead of pursuing the original approach by Leeb (which has very fruitfully been applied to a wide range of
Ramsey problems \cite{leeb-cat, Nesetril-Rodl, GLR})
we proposed in~\cite{masulovic-ramsey} a systematic study of
a simpler approach motivated by and implicit in \cite{vanthe-more,zucker,mu-pon}.
We have shown in~\cite{masulovic-ramsey} that the Ramsey property is a genuine categorical property since it
is preserved by categorical equivalence. Moreover,
right adjoints preserve the Ramsey property while left adjoins preserve the dual Ramsey property
(see \cite{masulovic-ramsey} for details). In Section~\ref{opos.sec.prelim} we
give a brief overview of standard notions referring to first order structures and category theory,
and conclude with the reinterpretation of the Ramsey property in the language of category theory.

In this paper we propose a new strategy to show that a class of structures has the Ramsey property.
The strategy is based on a relatively simple result in category theory and
consists of establishing a \emph{pre-adjunction} between the category
of structures which is known to have the Ramsey property, and the category of structures we are interested in.

There have been many attempts to weaken the notion of adjunction (see \cite{preadj-1, preadj-2})
as adjoint situations are extremely important not only in category theory but also in other mathematical
theories such as linear algebra and operator theory.
The version that we focus on in this paper will be referred to as a \emph{pre-adjunction}
and is defined in Section~\ref{opos.sec.PA}.
The main technical result in this paper is Theorem~\ref{opos.thm.main} which shows that ``right pre-adjoints''
preserve the Ramsey property. This is another confirmation of the fact that the Ramsey property is an extremely
robust categorical property.

Since pre-adjunctions represent rather loose relationships between categories, establishing a pre-adjunction
is much easier than establishing an adjunction or a categorical equivalence between two categories.
Thus, it turns out that Theorem~\ref{opos.thm.main} has a practical consequence: it
offers a ``piggyback'' strategy of proving that a category has the Ramsey property.
This strategy is motivated by the proof of~\cite[Theorem 12.13]{Promel-Book} where the Ramsey property for finite
linearly ordered graphs was shown as a consequence of the Graham-Rothschild Theorem.
In Section~\ref{opos.sec.cases} we demonstrate the applicability of the strategy based on pre-adjunctions 
by providing short proofs of three important well known results.
We first provide a straightforward proof of the Ramsey property for
the category $\catP$ of all finite linearly ordered posets with embeddings
by establishing a pre-adjunction with the Graham-Rothschild category $\catGR(\{0\}, X)$ (see Example~\ref{opos.ex.GR}
for the definition of the Graham-Rothschild category, and \cite{PTW, fouche} for the original proof that the class of linearly
ordered posets has the Ramsey property).
As ultrametric spaces are intimately related to posets (trees, actually), it comes as no surprise that
in the next step we provide a new proof that the category $\catU$ of finite convexly ordered ultrametric spaces with embeddings
has the Ramsey property (see~\cite{van-the-metric-spaces} for the original proof). In order to do so
we establish a pre-adjunction between the category $\catP$ and a family of full subcategories of 
$\catU$ which covers the entire $\catU$.
This idea can then be modified so as to provide a new proof of the fact that the categories $\catM$ and $\catM_\QQ$
of all finite linearly ordered metric spaces with embeddings, resp.\ all finite linearly ordered rational
metric spaces with embeddings, have the Ramsey property (see~\cite{Nesetril-metric} for the original proof).

The new method presented here accompanied by a few more transfer principles
generalizes to arbitrary classes of relational structures but it does not seem to generalize for classes with
forbidden substructures~\cite{masul-NRT}.

\section{Preliminaries}
\label{opos.sec.prelim}

In order to fix notation and terminology in this section we give a brief overview of standard notions referring to first order structures
and category theory, and conclude with the reinterpretation of the Ramsey property in the language of category theory.
For a systematic treatment of category-theoretic notions we refer the reader to~\cite{AHS}.

\subsection{Structures}

A \emph{structure} $\calA = (A, \Delta)$ is a set $A$ together with a set $\Delta$ of
functions and relations on $A$, each having some finite arity.
The underlying set of a structure $\calA$, $\calA_1$, $\calA^*$, \ldots\ will always be denoted by its roman letter $A$, $A_1$, $A^*$, \ldots\ respectively.
A structure $\calA = (A, \Delta)$ is \emph{finite} if $A$ is a finite set.

An \emph{embedding} $f: \calA \hookrightarrow \calB$ is an injection $f: A \rightarrow B$ which respects
the functions in $\Delta$, and respects and reflects the relations in $\Delta$.
Surjective embeddings are \emph{isomorphisms}. We write $\calA \cong \calB$ to denote that $\calA$ and $\calB$
are isomorphic, and $\calA \hookrightarrow \calB$ to denote that there is an embedding of $\calA$ into $\calB$.

A structure $\calA$ is a \emph{substructure} of a structure
$\calB$ ($\calA \le \calB$) if the identity map is an embedding of $\calA$ into $\calB$.
Let $\calA$ be a structure and $\0 \ne B \subseteq A$. Then $\restr \calA B = (B, \restr \Delta B)$ denotes
the \emph{substructure of $\calA$ induced by~$B$}, where $\restr \Delta B$ denotes the restriction of $\Delta$ to~$B$.
Note that $\restr \calA B$ is not required to exist for every $B \subseteq A$. For example, if $\Delta$ contains functions,
only those $B$ which are closed with respect to all the functions in $\Delta$ qualify for the base set of a substructure.

\begin{EX}\label{opos.ex.all}
  $(a)$
  A \emph{linearly ordered graph} is a structure $\calG = (V, E, \Boxed<)$ where $V$ is the set of vertices,
  $E \subseteq \big\{\{x, y\} : x, y \in V; x \ne y \big\}$ is the set of edges of $\calG$, and $<$ is a linear order on $V$.

  $(b)$
  A \emph{linearly ordered poset} is a structure $\calA = (A, \Boxed\sqsubseteq, \Boxed<)$ where
  $(A, \Boxed\sqsubseteq)$ is a poset and $<$ is a linear order on $A$ which extends $\sqsubseteq$
  (that is, if $a \sqsubseteq b$ and $a \ne b$ then $a < b$).

  $(c)$
  A \emph{linearly ordered metric space} is a structure $\calM = (M, d, \Boxed<)$ where $d : M^2 \to \RR$ is a
  \emph{metric} and $<$ is a linear order on $M$. A linearly ordered metric space $(M, d, \Boxed<)$ is \emph{rational} if
  $d : M^2 \to \QQ$.

  $(d)$
  A \emph{convexly ordered ultrametric space} is a structure $\calU = (U, d, \Boxed<)$ where $d : U^2 \to \RR$ is an
  \emph{ultrametric} (that is, a metric satisfying $d(x, z) \le \max\{d(x, y), d(y, z)\}$), and $<$ is a linear order on $U$
  such that every ball in $\calU$ is convex with respect to~$<$ (in other words, if $x, y \in B(u, r)$ and $x < z < y$ then
  $z \in B(u, r)$).
\end{EX}

Let $\calL = (L, \Boxed<)$ be a finite linearly ordered set.
For a nonempty $X \subseteq L$ let $\min_\calL(X)$, resp.\ $\max_\calL(X)$, denote the minimum, resp.\ maximum, of $X$ in $\calL$.
As a convention we let $\min_\calL \0 = \mathstrut$the top element of $\calL$, and
$\max_\calL \0 = \mathstrut$the bottom element of $\calL$.

Let $\lex$, $\alex$ and $\clex$ denote the \emph{lexicographic}, \emph{anti-lexicographic} and
\emph{complemented lexicographic} ordering on $\calP(L)$, respectively, defined as follows:
\begin{align*}
A \lex B \text{ iff }  A \subset B, &\text{ or }\min\nolimits_\calL(B \setminus A) < \min\nolimits_\calL(A \setminus B) \text{ in case}\\
                                           &\text{ $A$ and $B$ are incomparable};\\
A \alex B \text{ iff } A \subset B, &\text{ or }\max\nolimits_\calL(A \setminus B) < \max\nolimits_\calL(B \setminus A) \text{ in case}\\
                                           &\text{ $A$ and $B$ are incomparable};\\
A \clex B \text{ iff } A \supset B, &\text{ or }\min\nolimits_\calL(A \setminus B) < \min\nolimits_\calL(B \setminus A) \text{ in case}\\
                                            &\text{ $A$ and $B$ are incomparable}.
\end{align*}
(Note that $A \clex B$ iff $L \setminus A \lex L \setminus B$, hence the name.)

For $k \ge 1$ let $\lex$ and $\alex$ denote the \emph{lexicographic} and \emph{anti-lexicographic}
ordering on $L^k$, respectively, defined as follows:
\begin{align*}
(a_1, \ldots, a_k) \lex (b_1, \ldots, b_k) \text{ iff }
    & \text{there is an $i$ such that } a_i < b_i\\
    & \text{ and } a_j = b_j \text{ for all } j < i,\\
(a_1, \ldots, a_k) \alex (b_1, \ldots, b_k) \text{ iff }
    & \text{there is an $i$ such that } a_i < b_i\\
    & \text{ and } a_j = b_j \text{ for all } j > i.
\end{align*}

It is easy to see that all these are linear orders on $\calP(L)$ and $L^k$, respectively.

\subsection{Categories and functors}

In order to specify a \emph{category} $\CC$ one has to specify
a class of objects $\Ob(\CC)$, a set of morphisms $\hom_\CC(\calA, \calB)$ for all $\calA, \calB \in \Ob(\CC)$,
an identity morphism $\id_\calA$ for all $\calA \in \Ob(\CC)$, and
the composition of mor\-phi\-sms~$\cdot$~so that
$(f \cdot g) \cdot h = f \cdot (g \cdot h)$, and
$\id_\calB \cdot f = f \cdot \id_\calA$ for all $f \in \hom_\CC(\calA, \calB)$.
Instead of $\hom_\CC(\calA, \calB)$ we write $\hom(\calA, \calB)$ whenever $\CC$ is obvious from the context.

Let $\CC$ be a category. Every class $\KK \subseteq \Ob(\CC)$ can be turned into a category by
letting $\hom_\KK(\calA, \calB) = \hom_\CC(\calA, \calB)$
($\calA, \calB \in \KK$) and $f \cdot_{\KK} g = f \cdot_\CC g$.
We say then that $\KK$ is a \emph{full subcategory} of~$\CC$.

A \emph{functor} $F : \CC \to \DD$ from a category $\CC$ to a category $\DD$ maps $\Ob(\CC)$ to
$\Ob(\DD)$ and maps morphisms of $\CC$ to morphisms of $\DD$ so that
$F(f) \in \hom_\DD(F(\calA), F(\calB))$ whenever $f \in \hom_\CC(\calA, \calB)$, $F(f \cdot g) = F(f) \cdot F(g)$ whenever
$f \cdot g$ is defined, and $F(\id_\calA) = \id_{F(\calA)}$.

Categories $\CC$ and $\DD$ are \emph{isomorphic} if there exists a pair of functors
$F : \DD \rightleftarrows \CC : G$ mutually inverse on both objects and morphisms.

A pair of functors $F : \DD \rightleftarrows \CC : G$ is an \emph{adjunction} provided there is a family of isomorphisms
$\Phi_{\calY,\calX} : \hom_\CC(F(\calY), \calX) \cong \hom_\DD(\calY, G(\calX))$ indexed by pairs
$(\calY, \calX) \in \Ob(\DD) \times \Ob(\CC)$ which is natural in both $\calC$ and $\calD$.
We say that $F$ is \emph{left adjoint} to $G$ and $G$ is \emph{right adjoint} to $F$.
(For the full definition of adjunction the reader is referred to~\cite{AHS}.)

\begin{EX}
  Structures and some appropriately chosen morphisms usually constitute a category. For example,
  all finite linearly ordered graphs with embeddings constitute a category which we denote with $\catG$;
  all finite linearly ordered posets with embeddings constitute a category which we denote with $\catP$;
  all finite linearly ordered metric spaces with embeddings constitute a category which we denote with $\catM$; and
  all finite convexly ordered ultrametric spaces with embeddings constitute a category which we denote with $\catU$.

  For a linearly ordered metric space $\calM = (M, d, \Boxed<)$ let
  $$
    \spec(\calM) = \{d(x, y) : x, y \in M\}
  $$
  denote the \emph{spectre} of $\calM$, that is, the set of all the distances that are attained by points in~$\calM$.
  For a nonempty finite $S \subseteq \RR$ of nonnegative reals let $\catM_S$ denote the full subcategory of $\catM$
  spanned by all those $\calM \in \Ob(\catM)$ satisfying $\spec(\calM) \subseteq S$, and let
  $\catU_S$ denote the full subcategory of $\catU$ spanned by all those $\calU \in \Ob(\catU)$ satisfying $\spec(\calU) \subseteq S$.
\end{EX}

Let us now specify a category where objects are not structures and morphisms are not structure preserving maps.
We shall start from a class of well known structures and structure preserving maps, and in a few steps derive
an abstract category which captures all the relevant information about the original structures and their
structure preserving maps although its objects are positive integers and morphisms are not maps.
This abstraction process will be of particular importance for presenting what we call the Graham-Rothschild
category in Example~\ref{opos.ex.GR}.

\begin{EX}
  Let $F$ be a field. All finitely dimensional vector spaces over $F$ together with linear maps constitute
  a category that we shall denote with $\catV_F$. Since every finitely dimensional vector space over $F$ is isomorphic
  to $F^n$ for some $n$, in some cases it may be more convenient to consider the category $\catV'_F$ whose objects
  are $\{F^n : n \in \NN\}$ and whose morphisms are all the linear maps $F^k \to F^n$ where $k, n \in \NN$.
  The category $\catV'_F$ is a full subcategory of $\catV_F$ and, up to isomorphism, every object of $\catV_F$
  is uniquely represented in $\catV'_F$. Such a subcategory is referred to as the \emph{skeleton} of the original
  category.
  
  Keeping the field $F$ fixed, what really matters when considering finite dimensional vector spaces over $F$ is the
  dimension of the space. Moreover, by fixing in each $F^n$ the standard base
  $e_i = (0, \ldots, 0, \place{i}{th}{1}, 0, \ldots, 0)$, $1 \le i \le n$, we can uniquely represent
  linear maps $F^k \to F^n$ by $n \times k$ matrices over $F$. So, let $\catV''_F$ denote the category whose
  objects are the positive integers $1, 2, 3, \ldots$ and each homset $\hom(k, n)$ consists of all the
  $n \times k$ matrices over $F$. The composition of morphisms is realized by matrix multiplication.
  Clearly, the categories $\catV'_F$ and $\catV''_F$ are isomorphic so the abstract category $\catV''_F$
  captures all the relevant information about finitely dimensional vector spaces over~$F$.
\end{EX}

The following example introduces what we call the Graham-Rothschild category which will be
of particular importance in the sequel.

\begin{EX}\label{opos.ex.GR}
  Let $A$ be a finite alphabet. A word $u$ of length $n$ over $A$ can be thought of as
  an element of $A^n$ but also as a mapping $u : \{1, 2, \ldots, n\} \to A$. In the latter case
  $u^{-1}(a)$, $a \in A$, denotes the set of all the positions in $u$ where $a$ appears.

  Let $X = \{x_1, x_2, \ldots\}$ be a countably infinite set of variables and let
  $A$ be a finite alphabet disjoint from $X$. An \emph{$m$-parameter word over $A$ of length $n$} is a word
  $w \in (A \union \{x_1, x_2, \ldots, x_m\})^n$ satisfying the following:
  \begin{itemize}
  \item
    each of the letters $x_1, \ldots, x_m$ appears at least once in $w$, and
  \item
    $\min w^{-1}(x_i) < \min w^{-1}(x_j)$ whenever $1 \le i < j \le m$.
  \end{itemize}
  Let $W^n_m(A)$ denote the set of all the $m$-parameter words over $A$ of length~$n$.
  For $u \in W^n_m(A)$ and $v = v_1 v_2 \ldots v_m \in W^m_k(A)$ let
  $$
    u \cdot v = u[v_1/x_1, v_2/x_2, \ldots, v_m/x_m] \in W^n_k(A)
  $$
  denote the word obtained by replacing each occurrence of $x_i$ in $u$ with $v_i$,
  simultaneously for all $i \in \{1, \ldots, m\}$.
  
  Let $\catGR(A, X)$ denote the \emph{Graham-Rothschild category over $A$ and $X$} whose objects are positive integers 1, 2, \ldots,
  whose morphisms are given by $\hom(k, n) = W^n_k(A)$ if $k \le n$ and $\hom(k, n) = \0$ if $k > n$, where the composition
  is $\cdot$ defined above and the identity morphism $\id_n$ is given by $x_1 x_2 \ldots x_n$.
\end{EX}

\subsection{Structural Ramsey property in the language of category theory}

Let $\CC$ be a category and $\calS$ a set. We say that
$
  \calS = \Sigma_1 \union \ldots \union \Sigma_k
$
is a \emph{$k$-coloring} of $\calS$ if $\Sigma_i \sec \Sigma_j = \0$ whenever $i \ne j$.
For an integer $k \ge 2$ and $\calA, \calB, \calC \in \Ob(\CC)$ we write
$
  \calC \longrightarrow (\calB)^{\calA}_k
$
to denote that for every $k$-coloring
$
  \hom_\CC(\calA, \calC) = \Sigma_1 \union \ldots \union \Sigma_k
$
there is an $i \in \{1, \ldots, k\}$ and a morphism $w \in \hom_\CC(\calB, \calC)$ such that
$w \cdot \hom_\CC(\calA, \calB) \subseteq \Sigma_i$.

A category $\CC$ has the \emph{Ramsey property} if
for every integer $k \ge 2$ and all $\calA, \calB \in \Ob(\CC)$
such that $\hom_\CC(\calA, \calB) \ne \0$ there is a
$\calC \in \Ob(\CC)$ such that $\calC \longrightarrow (\calB)^{\calA}_k$.

\begin{EX}\label{opos.ex.RP}
  The categories $\catG$, $\catP$, $\catM$, $\catM_\QQ$ and $\catU$ have the Ramsey property.
  Note that this is just a reformulation of the well known results
  proved in \cite{AH, Nesetril-Rodl-1976}, \cite{PTW, fouche}, \cite{Nesetril-metric}, \cite{Nesetril-metric}
  and \cite{van-the-metric-spaces}, respectively.
  
  For every finite set $A$ and a countably infinite set $X = \{x_1, x_2, \ldots\}$ disjoint from $A$ the Graham-Rothschild
  category $\catGR(A, X)$ has the Ramsey property. This is just a reformulation of the famous Graham-Rothschild Theorem:

 \begin{THM} \cite{GR}\label{opos.thm.G-R}
    Let $A$ be a finite alphabet and let
    $m, \ell \ge 1$ and $k \ge 2$. Then there exists an $n$ such that for every partition
    $W^n_\ell(A) = \Sigma_1 \union \ldots \union \Sigma_k$ there exist a $u \in W^n_m(A)$ and $j$ such that
    $\{u \cdot v : v \in W^m_\ell(A)\} \subseteq \Sigma_j$.
  \end{THM}
\end{EX}

\section{The Ramsey property and pre-adjunctions}
\label{opos.sec.PA}

There have been many attempts to weaken the notion of adjunction \cite{preadj-1, preadj-2}.
In this paper we consider the following version that we refer to as a pre-adjunction.

\begin{figure}
  $$
    \XYMATRIX{
      F(\calD) \ar[rr]^u                     & & \calC & & \calD \ar[rr]^{\Phi_{\calD, \calC}(u)}                    & & G(\calC) \\
      F(\calE) \ar[u]^v \ar[urr]_{u \cdot v} & &       & & \calE \ar[u]^f \ar[urr]_{\Phi_{\calE, \calC}(u \cdot v)}
    }
  $$
\caption{The requirement (PA)}
\label{opos.fig.PA}
\end{figure}

\begin{DEF}\label{opos.def.PA}
  Let $\CC$ and $\DD$ be categories. A pair of maps
  $$
    F : \Ob(\DD) \rightleftarrows \Ob(\CC) : G
  $$
  is a \emph{pre-adjunction between $\CC$ and $\DD$} provided there is a family of maps
  $$
    \Phi_{\calY,\calX} : \hom_\CC(F(\calY), \calX) \to \hom_\DD(\calY, G(\calX))
  $$
  indexed by the family $\{(\calY, \calX) \in \Ob(\DD) \times \Ob(\CC) : \hom_\CC(F(\calY), \calX) \ne \0\}$
  and satisfying the following (see Fig.~\ref{opos.fig.PA}):
  \begin{itemize}
  \item[(PA)]
  for every $\calC \in \Ob(\CC)$, every $\calD, \calE \in \Ob(\DD)$,
  every $u \in \hom_\CC(F(\calD), \calC)$ and every $f \in \hom_\DD(\calE, \calD)$ there is a $v \in \hom_\CC(F(\calE), F(\calD))$
  satisfying $\Phi_{\calD, \calC}(u) \cdot f = \Phi_{\calE, \calC}(u \cdot v)$.
  \end{itemize}
\end{DEF}
(Note that in a pre-adjunction $F$ and $G$ are \emph{not} required to be functors, just maps from the class of objects of one of the two
categories into the class of objects of the other category; also $\Phi$ is not required to be a natural isomorphism, just a family of
maps between hom-sets satisfying the requirement above.)

\begin{THM}\label{opos.thm.main}
  Let $\CC$ and $\DD$ be categories and let $F : \Ob(\DD) \rightleftarrows \Ob(\CC) : G$ be a pre-adjunction with
  $\Phi_{\calY,\calX} : \hom_\CC(F(\calY), \calX) \to \hom_\DD(\calY, G(\calX))$ as the corresponding
  family of maps between hom-sets. Assume that $\CC$ has the Ramsey property. Then $\DD$ has the Ramsey property.
\end{THM}
\begin{proof}
  Take any $\calD, \calE \in \Ob(\DD)$ and an integer $k \ge 2$. Since $\CC$ has the Ramsey property, there is a $\calC \in \Ob(\CC)$
  such that $\calC \longrightarrow (F(\calD))^{F(\calE)}_k$. Let us show that $G(\calC) \longrightarrow (\calD)^{\calE}_k$.
  Take any coloring $\hom_\DD(\calE, G(\calC)) = \Sigma_1 \union \ldots \union \Sigma_k$ and construct a coloring
  $\hom_\CC(F(\calE), \calC) = \Sigma'_1 \union \ldots \union \Sigma'_k$ as follows:
  \begin{equation}\label{opos.eq.1}
    \Sigma'_i = \{u \in \hom_\CC(F(\calE), \calC) : \Phi_{\calE, \calC}(u) \in \Sigma_i\}.
  \end{equation}
  By the choice of $\calC$ there is a $u \in \hom_\CC(F(\calD), \calC)$ and a $j \in \{1, \ldots, k\}$ such that
  \begin{equation}\label{opos.eq.2}
    u \cdot \hom_\CC(F(\calE), F(\calD)) \subseteq \Sigma'_j.
  \end{equation}
  Let us show that
  \begin{equation}\label{opos.eq.3}
    \Phi_{\calD, \calC}(u) \cdot \hom_\DD(\calE, \calD) \subseteq \Sigma_j.
  \end{equation}
  Take any $f \in \hom_\DD(\calE, \calD)$. Since $F : \Ob(\DD) \rightleftarrows \Ob(\CC) : G$ is a pre-adjunction, there is a
  $v \in \hom_\CC(F(\calE), F(\calD))$ such that
  \begin{equation}\label{opos.eq.4}
    \Phi_{\calD, \calC}(u) \cdot f = \Phi_{\calE, \calC}(u \cdot v).
  \end{equation}
  By \eqref{opos.eq.2} we have that $u \cdot v \in \Sigma'_j$, so $\Phi_{\calE, \calC}(u \cdot v) \in \Sigma_j$ by
  \eqref{opos.eq.1}. This, together with \eqref{opos.eq.4} yields $\Phi_{\calD, \calC}(u) \cdot f  \in \Sigma_j$.
\end{proof}

This strategy is motivated by the proof of~\cite[Theorem 12.13]{Promel-Book} where the Ramsey property for finite
linearly ordered graphs was shown as a consequence of the Graham-Rothschild Theorem.
We shall now provide the same proof in the new parlance introduced above. In order to motivate the main idea
that might seem odd at first reading, and in order to explain the
underlying combinatorial idea of representing graphs by parametric words and
special subgraphs by subspaces we shall demonstrate the key aspects of the construction on a small example.

\begin{THM}\label{opos.thm.GRA}
  The category $\catG$ has the Ramsey property.
\end{THM}
\begin{proof} (cf.\ \cite[Theorem 12.13]{Promel-Book})
  It suffices to show that there is a pre-adjunction
  $$
    F : \Ob(\catG) \rightleftarrows \Ob(\catGR(\{0\}, X)) : G,
  $$
  where $X$ is a countably infinite set of variables disjoint from $\{0\}$.
  The result then follows from Theorem~\ref{opos.thm.main} and the fact that the category $\catGR(\{0\}, X)$ has the Ramsey property
  (Example~\ref{opos.ex.RP}).

  For a $\calG = (V, E, \Boxed<) \in \Ob(\catG)$ let $F(\calG) = |V| + |E|$. On the other hand, for a
  positive integer $n$ (recall that $\Ob(\catGR(A, X)) = \NN$) let
  $$
    G(n) = (\calP(\{1, 2, \ldots, n\}), E_n, \Boxed\clex)
  $$
  denote the finite linearly ordered graph on $\calP(\{1, 2, \ldots, n\})$
  where $\{X, Y\} \in E_n$ if and only if $X \sec Y \ne \0$,
  and $\clex$ is the complemented lexicographic ordering of $\calP(\{1, 2, \ldots, N\})$ induced by
  the usual ordering of the integers.

  For a finite linearly ordered graph $\calG$ and a positive integer $N$ define
  $$
    \Phi_{\calG, N} : \hom_{\catGR(\{0\}, X)}(F(\calG), N) \to \hom_{\catG}(\calG, G(N))
  $$
  as follows. Let $\calG = (V, E, \Boxed<)$, where $V = \{v_1 < \ldots < v_n\}$ and
  $E = \{e_1 \clex \ldots \clex e_m\}$. Take any $u \in \hom_{\catGR(\{0\}, X)}(n + m, N) =
  W^N_{n + m}(\{0\})$ and for every $i \in \{1, \ldots, n + m\}$ let $X_i = u^{-1}(x_i)$. For $i \in \{1, \ldots, n\}$
  let
  $$
    \tilde v_i = X_i \union \UNION_{j \; : \; e_j \ni v_i} X_{n + j}.
  $$
  It is easy to see that $\hat u : \calG \to G(N) : v_i \mapsto \tilde v_i$ is an embedding of linearly ordered
  graphs (note that, by construction, $v_i$ and $v_j$ are adjacent if and only if $\tilde v_i \sec \tilde v_j \ne \0$),
  so we put $\Phi_{\calG, N}(u) = \hat u$.

  \medskip

  \textit{Example.}
  Let us demonstrate the above construction by means of a small example.
  Let $\calG = (\{1, 2, 3, 4\}, \{e_1, e_2, e_3\}, \Boxed<)$ be the following graph where $<$ is the usual
  ordering of the integers, and $e_1 = \{1,2\}$, $e_2 = \{2, 3\}$ and $e_3 = \{2, 4\}$:
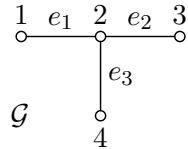
  \begin{center}
  \begin{pgfpicture}
    \pgfsetxvec{\pgfpoint{\acadpgfunit}{0pt}}
    \pgfsetyvec{\pgfpoint{0pt}{\acadpgfunit}}
    \pgfsetlinewidth{\acadpgflinewidth}
    \pgftransformshift{\pgfpointxy{-87.5}{-187.5}}
  
    \begin{pgfscope}
      \pgfpathmoveto{\pgfpointxy{150.0}{400.0}}
      \pgfpathlineto{\pgfpointxy{275.0}{400.0}}
      \pgfusepath{stroke}
    \end{pgfscope}
    \begin{pgfscope}
      \pgfpathmoveto{\pgfpointxy{275.0}{400.0}}
      \pgfpathlineto{\pgfpointxy{400.0}{400.0}}
      \pgfusepath{stroke}
    \end{pgfscope}
    \begin{pgfscope}
      \pgfpathmoveto{\pgfpointxy{275.0}{400.0}}
      \pgfpathlineto{\pgfpointxy{275.0}{275.0}}
      \pgfusepath{stroke}
    \end{pgfscope}
    \begin{pgfscope}
      \pgfsetfillcolor{white}
      \pgfpathellipse{\pgfpointxy{150.0}{400.0}}{\pgfpointxy{8.0}{0.0}}{\pgfpointxy{0.0}{8.0}}
      \pgfusepath{fill,stroke}
    \end{pgfscope}
    \begin{pgfscope}
      \pgfsetfillcolor{white}
      \pgfpathellipse{\pgfpointxy{275.0}{400.0}}{\pgfpointxy{8.0}{0.0}}{\pgfpointxy{0.0}{8.0}}
      \pgfusepath{fill,stroke}
    \end{pgfscope}
    \begin{pgfscope}
      \pgfsetfillcolor{white}
      \pgfpathellipse{\pgfpointxy{400.0}{400.0}}{\pgfpointxy{8.0}{0.0}}{\pgfpointxy{0.0}{8.0}}
      \pgfusepath{fill,stroke}
    \end{pgfscope}
    \begin{pgfscope}
      \pgfsetfillcolor{white}
      \pgfpathellipse{\pgfpointxy{275.0}{275.0}}{\pgfpointxy{8.0}{0.0}}{\pgfpointxy{0.0}{8.0}}
      \pgfusepath{fill,stroke}
    \end{pgfscope}
    \pgftext[bottom,at={\pgfpointxy{150.0}{420.0}}]{1}
    \pgftext[bottom,at={\pgfpointxy{275.0}{420.0}}]{2}
    \pgftext[bottom,at={\pgfpointxy{400.0}{420.0}}]{3}
    \pgftext[top,at={\pgfpointxy{275.0}{255.0}}]{4}
    \pgftext[bottom,at={\pgfpointxy{212.5}{412.0}}]{$e_1$}
    \pgftext[bottom,at={\pgfpointxy{337.5}{412.0}}]{$e_2$}
    \pgftext[left,at={\pgfpointxy{287.0}{337.5}}]{$e_3$}
    \pgftext[bottom,left,at={\pgfpointxy{133.0}{258.0}}]{$\calG$}
  \end{pgfpicture}
  \end{center}
  Since $\calG$ has 4 vertices and 3 edges, we consider 7-parameter words of arbitrary length $N$ over the alphabet $\{0\}$.
  Let $N = 16$ and let
  $$
    u = 0 x_1 0 0 x_2 0 x_1 x_3 x_3 x_4 x_2 x_5 x_6 0 x_7 x_1 \in W^{16}_7(\{0\}).
  $$
  Then $X_1 = u^{-1}(x_1) = \{2, 7, 16\}$, $X_2 = u^{-1}(x_2) = \{5, 11\}$, $X_3 = u^{-1}(x_3) = \{8, 9\}$ and
  $X_4 = u^{-1}(x_4) = \{10\}$ correspond to the vertices of $\calG$, while
  $X_5 = u^{-1}(x_5) = \{12\}$, $X_6 = u^{-1}(x_6) = \{13\}$ and $X_7 = u^{-1}(x_7) = \{15\}$
  correspond to the edges of $\calG$. Using $X_1, X_2, \ldots, X_7$ we now construct a copy of $\calG$ in $G(16)$ as
  follows:
  $$
    \begin{array}{@{}l@{\quad}l@{}}
      \tilde v_1 = X_1 \union X_5 = \{2, 7, 16, 12\}, & (\text{1 is incident with } e_1),\\
      \tilde v_2 = X_2 \union X_5 \union X_6 \union X_7 = \{5, 11, 12, 13, 15\},    & (\text{2 is incident with } e_1, e_2, e_3),\\
      \tilde v_3 = X_3 \union X_6 = \{8, 9, 13\},     & (\text{3 is incident with } e_2),\\
      \tilde v_4 = X_4 \union X_7 = \{10, 15\},       & (\text{4 is incident with } e_3).
    \end{array}
  $$
  Clearly, $i$ and $j$ are adjacent in $\calG$ if and only if $\tilde v_i \sec \tilde v_j \ne \0$. Moreover,
  $\tilde v_1 \clex \tilde v_2 \clex \tilde v_3 \clex \tilde v_4$,
  whence follows that the subgraph of $G(16)$ induced by $\tilde v_1$, $\tilde v_2$, $\tilde v_3$ and $\tilde v_4$
  is isomorphic to~$\calG$.

  \medskip

  Going back to the proof,
  let $\calG' = (V', E', \Boxed{<'})$, where $V' = \{v'_1 \mathrel{<'} \ldots \mathrel{<'} v'_p\}$ and
  $E' = \{e'_1 \clexprime \ldots \clexprime e'_q\}$. Let $f : \calG' \hookrightarrow \calG$ be an
  embedding. For $j \in \{1, \ldots, q\}$ let
  $$
    X'_{p + j} = \hat u(f(v'_i)) \sec \hat u(f(v'_k)) \text{ where } e'_j = \{ v'_i, v'_k \},
  $$
  and then for $i \in \{1, \ldots, p\}$ put
  $$
    X'_i = \hat u(f(v'_i)) \setminus \UNION_{j =1}^q X'_{p + j}.
  $$
  Finally, define $h = h_1 h_2 \ldots h_{n + m} \in \hom_{\catGR(\{0\}, X)}(p + q, n + m) = W^{n + m}_{p + q}(\{0\})$ as follows:
  $$
    h_j = \begin{cases}
      x_i, & X_j \subseteq X'_i\\
      0,   & \text{otherwise.}
    \end{cases}
  $$
  Then it is a routine to check that $\Phi_{\calG, N}(u) \circ f = \Phi_{\calG', N}(u \cdot h)$.
  
  \medskip
  
  \textit{Example (continued).}
  Let us conclude the proof by illustrating the final step in the construction.
  Let $\calG'$ be a graph which embeds into $\calG$ via
  $f = \begin{pmatrix}1 & 2 & 3\\2 & 3 & 4\end{pmatrix}$:
  
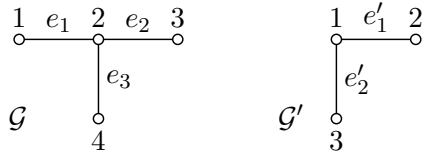
\begin{center}
  \begin{pgfpicture}
    \pgfsetxvec{\pgfpoint{\acadpgfunit}{0pt}}
    \pgfsetyvec{\pgfpoint{0pt}{\acadpgfunit}}
    \pgfsetlinewidth{\acadpgflinewidth}
    \pgftransformshift{\pgfpointxy{-87.5}{-175.0}}
  
    \begin{pgfscope}
      \pgfpathmoveto{\pgfpointxy{150.0}{400.0}}
      \pgfpathlineto{\pgfpointxy{275.0}{400.0}}
      \pgfusepath{stroke}
    \end{pgfscope}
    \begin{pgfscope}
      \pgfpathmoveto{\pgfpointxy{275.0}{400.0}}
      \pgfpathlineto{\pgfpointxy{400.0}{400.0}}
      \pgfusepath{stroke}
    \end{pgfscope}
    \begin{pgfscope}
      \pgfpathmoveto{\pgfpointxy{275.0}{400.0}}
      \pgfpathlineto{\pgfpointxy{275.0}{275.0}}
      \pgfusepath{stroke}
    \end{pgfscope}
    \begin{pgfscope}
      \pgfpathmoveto{\pgfpointxy{650.0}{400.0}}
      \pgfpathlineto{\pgfpointxy{650.0}{275.0}}
      \pgfusepath{stroke}
    \end{pgfscope}
    \begin{pgfscope}
      \pgfpathmoveto{\pgfpointxy{650.0}{400.0}}
      \pgfpathlineto{\pgfpointxy{775.0}{400.0}}
      \pgfusepath{stroke}
    \end{pgfscope}
    \begin{pgfscope}
      \pgfsetfillcolor{white}
      \pgfpathellipse{\pgfpointxy{150.0}{400.0}}{\pgfpointxy{8.0}{0.0}}{\pgfpointxy{0.0}{8.0}}
      \pgfusepath{fill,stroke}
    \end{pgfscope}
    \begin{pgfscope}
      \pgfsetfillcolor{white}
      \pgfpathellipse{\pgfpointxy{275.0}{400.0}}{\pgfpointxy{8.0}{0.0}}{\pgfpointxy{0.0}{8.0}}
      \pgfusepath{fill,stroke}
    \end{pgfscope}
    \begin{pgfscope}
      \pgfsetfillcolor{white}
      \pgfpathellipse{\pgfpointxy{400.0}{400.0}}{\pgfpointxy{8.0}{0.0}}{\pgfpointxy{0.0}{8.0}}
      \pgfusepath{fill,stroke}
    \end{pgfscope}
    \begin{pgfscope}
      \pgfsetfillcolor{white}
      \pgfpathellipse{\pgfpointxy{275.0}{275.0}}{\pgfpointxy{8.0}{0.0}}{\pgfpointxy{0.0}{8.0}}
      \pgfusepath{fill,stroke}
    \end{pgfscope}
    \begin{pgfscope}
      \pgfsetfillcolor{white}
      \pgfpathellipse{\pgfpointxy{650.0}{400.0}}{\pgfpointxy{8.0}{0.0}}{\pgfpointxy{0.0}{8.0}}
      \pgfusepath{fill,stroke}
    \end{pgfscope}
    \begin{pgfscope}
      \pgfsetfillcolor{white}
      \pgfpathellipse{\pgfpointxy{650.0}{275.0}}{\pgfpointxy{8.0}{0.0}}{\pgfpointxy{0.0}{8.0}}
      \pgfusepath{fill,stroke}
    \end{pgfscope}
    \begin{pgfscope}
      \pgfsetfillcolor{white}
      \pgfpathellipse{\pgfpointxy{775.0}{400.0}}{\pgfpointxy{8.0}{0.0}}{\pgfpointxy{0.0}{8.0}}
      \pgfusepath{fill,stroke}
    \end{pgfscope}
    \pgftext[bottom,at={\pgfpointxy{150.0}{420.0}}]{1}
    \pgftext[bottom,at={\pgfpointxy{275.0}{420.0}}]{2}
    \pgftext[bottom,at={\pgfpointxy{400.0}{420.0}}]{3}
    \pgftext[top,at={\pgfpointxy{275.0}{255.0}}]{4}
    \pgftext[bottom,at={\pgfpointxy{212.5}{412.0}}]{$e_1$}
    \pgftext[bottom,at={\pgfpointxy{337.5}{412.0}}]{$e_2$}
    \pgftext[left,at={\pgfpointxy{287.0}{337.5}}]{$e_3$}
    \pgftext[bottom,left,at={\pgfpointxy{133.0}{258.0}}]{$\calG$}
    \pgftext[bottom,at={\pgfpointxy{650.0}{420.0}}]{1}
    \pgftext[top,at={\pgfpointxy{650.0}{255.0}}]{3}
    \pgftext[left,at={\pgfpointxy{662.0}{337.5}}]{$e'_2$}
    \pgftext[bottom,left,at={\pgfpointxy{558.0}{258.0}}]{$\calG'$}
    \pgftext[bottom,at={\pgfpointxy{775.0}{420.0}}]{2}
    \pgftext[bottom,at={\pgfpointxy{712.5}{412.0}}]{$e'_1$}
  \end{pgfpicture}
  \end{center}
  Since $\calG'$ has 3 vertices and 2 edges we are looking for an $h \in W^{7}_{5}(\{0\})$. Let us first construct
  $X'_4$ and $X'_5$, and then $X'_1$, $X'_2$ i $X'_3$:
  \begin{align*}
    X'_4 &= \hat u(f(1)) \sec \hat u(f(2)) = \tilde v_2 \sec \tilde v_3 = \{ 13 \},\\
    X'_5 &= \hat u(f(1)) \sec \hat u(f(3)) = \tilde v_2 \sec \tilde v_4 = \{ 15 \},\\
    X'_1 &= \hat u(f(1)) \setminus (X'_4 \union X'_5) = \tilde v_2 \setminus \{ 13, 15 \} = \{ 5, 11, 12 \},\\
    X'_2 &= \hat u(f(2)) \setminus (X'_4 \union X'_5) = \tilde v_3 \setminus \{ 13, 15 \} = \{ 8, 9 \},\\
    X'_3 &= \hat u(f(3)) \setminus (X'_4 \union X'_5) = \tilde v_4 \setminus \{ 13, 15 \} = \{ 10 \}.
  \end{align*}
  The word $h = h_1 h_2 h_3 h_4 h_5 h_6 h_7$ can then be computed as follows:
      $X_2 \subseteq X'_1 \Rightarrow h_2 = x_1$,
      $X_5 \subseteq X'_1 \Rightarrow h_5 = x_1$,
      $X_3 \subseteq X'_2 \Rightarrow h_3 = x_2$,
      $X_4 \subseteq X'_3 \Rightarrow h_4 = x_3$,
      $X_6 \subseteq X'_4 \Rightarrow h_6 = x_4$,
      $X_7 \subseteq X'_5 \Rightarrow h_7 = x_5$,
  while $h_i = 0$ for the remaining $i$'s. Therefore, $h = 0 x_1 x_2 x_3 x_1 x_4 x_5$.

  \medskip

  Finally, let us show that $\Phi_{\calG, N}(u) \circ f = \Phi_{\calG', N}(u \cdot h)$.
  Clearly, $u \cdot h = 0 0 0 0 x_1 0 0 x_2 x_2 x_3 x_1 x_1 x_4 0 x_5 0$.
  To compute $\Phi_{\calG', N}(u \cdot h)$ note, first, that
  $X_1 = u^{-1}(x_1) = \{5, 11, 12\}$,
  $X_2 = u^{-1}(x_2) = \{8, 9\}$,
  $X_3 = u^{-1}(x_3) = \{10\}$,
  $X_4 = u^{-1}(x_4) = \{13\}$ and
  $X_5 = u^{-1}(x_5) = \{15\}$,
  so the copy of $\calG'$ in $G(16)$ we have thus encoded is:
  $$
    \begin{array}{l@{\quad}l}
      \tilde w_1 = X_1 \union X_4 \union X_5 = \{5, 11, 12, 13, 15\}, & (\text{1 is incident with } e'_1, e'_2),\\
      \tilde w_2 = X_2 \union X_4 = \{8, 9, 13\},    & (\text{2 is incident with } e'_1),\\
      \tilde w_3 = X_3 \union X_5 = \{10, 15\},     & (\text{3 is incident with } e'_2).
    \end{array}
  $$
  But this is exactly the subgraph of $\Phi_{\calG, N}(u)$ induced by the vertices $\tilde v_2, \tilde v_3, \tilde v_4$.
\end{proof}

This result is one of the first results in structural Ramsey theory. It was proved independently and in different contexts
in \cite{AH} and \cite{Nesetril-Rodl-1976}. A new proof based on a variation of the amalgamation technique called
the \emph{partite construction} was published in~\cite{Nesetril-Rodl-1989}. In contrast to these proofs which,
although based on different approaches, given $\calA$, $\calB$ and $k \ge 2$ explicitly construct $\calC$ such that
$\calC \longrightarrow (\calB)^\calA_k$, the proof presented in \cite[Theorem 12.13]{Promel-Book} takes a radically
different point of view: it simply encodes the context we are interested in into the context where Ramsey property has already
been proven. Our principal insight is that this approach can be turned into a general strategy for
transferring the Ramsey property from a context to another context. In the section that follows
we demonstrate the applicability of the strategy based on pre-adjunctions 
by providing short proofs of three important well known results.

Interestingly, the idea of the proof of Theorem~\ref{opos.thm.GRA} does not extend to (uniform) hypergraphs.
In the argument verifying the construction of the parametric word we essentially rely on the fact that every edge
is either part of the subgraph (and thus corresponds to an edge of the subgraph) or it has at most one end in the subgraph
and both these cases can be represented by corresponding parameter.

\section{Linearly ordered posets and metric spaces}
\label{opos.sec.cases}

In this section we apply the strategy based on pre-adjunctions
to provide new straightforward proofs of three important well known results.
We are first going to show that the category $\catP$ of all finite linearly ordered posets with embeddings has the Ramsey property
by establishing a pre-adjunction with the Graham-Rothschild category $\catGR(\{0\}, X)$ (see~\cite{PTW, fouche}
for the original proof). Interestingly, the category $\catP$ was the only known Ramsey class of structures where
the proof of the Ramsey property relied on proving first that the class has the ordering property.
We shall not define the ordering property here (for the definition and a detailed discussion we refer the reader to~\cite{Nesetril}),
but let us mention that the ordering property is a property related to the Ramsey property
which usually follows from the fact that the class under consideration has the Ramsey property.
The proof we present here is new not only because we use new proof strategies, but also because it does not
not rely on the ordering property. The ordering property can now be shown to follow from the Ramsey property for $\catP$
(the proof will appear elsewhere) and thus the aberration of the category $\catP$ has been rectified.

We then provide a new proof that the category $\catU$ of finite convexly ordered ultrametric spaces with embeddings
has the Ramsey property (see~\cite{van-the-metric-spaces} for the original proof). In order to do so
we establish a pre-adjunction between the category $\catP$ and a family of full subcategories of 
$\catU$ which covers the entire $\catU$.

This idea can then be modified so as to provide a new proof of the fact that the category $\catM$ of all finite
linearly ordered metric spaces with embeddings has the Ramsey property by establishing
a pre-adjunction between the category $\catP$ and a family of full subcategories of 
$\catM$ which covers the entire $\catM$ (see~\cite{Nesetril-metric} for the original proof).

\begin{THM}\label{opos.thm.1} (cf.\ \cite{PTW, fouche})
  The category $\catP$ has the Ramsey property.
\end{THM}
\begin{proof}
  It suffices to show that there is a pre-adjunction
  $$
    F : \Ob(\catP) \rightleftarrows \Ob(\catGR(\{0\}, X)) : G,
  $$
  where $X$ is a countably infinite set of variables disjoint from $\{0\}$.
  The result then follows from Theorem~\ref{opos.thm.main} and the fact that the category $\catGR(\{0\}, X)$ has the Ramsey property
  (Example~\ref{opos.ex.RP}).

  Recall that a \emph{downset} in a finite linearly ordered poset $\calA = (A, \Boxed\sqsubseteq, \Boxed<)$
  is a subset $B \subseteq A$ such that $x \in B$ and $y \sqsubseteq x$ implies $y \in B$.
  For $a \in A$ let $\downarrow_\calA a = \{b \in A : b \sqsubseteq a\}$.
  Clearly, $\downarrow_\calA a$ is always a downset in $\calA$, but not all the downsets in a poset are of the form
  $\downarrow_\calA a$. To see this, take two incomparable elements $a, b \in A$. Then $\downarrow_\calA a \;\union \downarrow_\calA b$
  is a donwset in $\calA$ which is not of the form $\downarrow_\calA x$ for some $x \in A$.
  
  For an $\calA = (A, \Boxed\sqsubseteq, \Boxed<) \in \Ob(\catP)$ let
  $F(\calA) = \mathstrut$the number of distinct nonempty downsets in $(A, \Boxed\sqsubseteq)$.
  On the other hand, for a positive integer $n$ let $G(n) = \big(\calP(\{1, \ldots, n\}), \Boxed\supseteq, \Boxed\clex\big)$.
  Clearly, $G(n)$ is a linearly ordered poset.
  
  For a finite linearly ordered poset $\calA$ and a positive integer $n$ define
  $$
    \Phi_{\calA, n} : \hom_{\catGR(\{0\}, X)}(F(\calA), n) \to \hom_{\catP}(\calA, G(n))
  $$
  as follows. Let $\calA = (\{1, 2, \ldots, k\}, \Boxed\sqsubseteq, \Boxed<)$ where $<$ is the usual
  ordering of the integers. Let $D_1$, \ldots, $D_m$ be all the nonempty downsets in $\calA$ and let
  $D_1 \alex D_2 \alex \ldots \alex D_m$.
  For $u \in \hom_{\catGR(\{0\}, X)}(F(\calA), n) = W^n_m(\{0\})$, let $X_i = u^{-1}(x_i)$, $1 \le i \le m$, and let
  $a_i = \UNION\{X_\alpha : i \in D_\alpha\}$, $1 \le i \le k$. Put $\Phi_{\calA, n}(u) = \phi_u$
  where $\phi_u : \calA \to G(n) : i \mapsto a_i$.

  Let us pause for a moment to explain why we need to order the downsets of $\calA$ anti-lexicographically.
  In the proof of Theorem~\ref{opos.thm.GRA}, which was the principal motivation for this proof,
  it is important to order the edges of a graph somehow -- the choice of the actual ordering relation
  is irrelevant as long as this can be done systematically. Here, however, we have to be more careful.
  In order to ensure that it is always possible to construct a parametric word for a given subposet
  the ordering relation has to be chosen so that it does not change for subposets, and
  it orders downsets so that the downset of the vertex appear first.

  Going back to the proof, let us show that the definition of $\Phi$ is correct by showing
  that for every $u \in W^n_m(\{0\})$ the mapping $\phi_u$ is an embedding $\calA \hookrightarrow G(n)$.

  Let us first show that $i \sqsubseteq j$ implies $a_i \supseteq a_j$. Recall that
  $a_i = \UNION\{X_\alpha : i \in D_\alpha\}$ and $a_j = \UNION\{X_\beta : j \in D_\beta\}$.
  Take any $X_\beta \subseteq a_j$. Then $j \in D_\beta$ so $i \sqsubseteq j$ and the fact that
  $D_\beta$ is a downset imply $i \in D_\beta$. Therefore, $X_\beta \subseteq a_i$.
  
  Assume, next, that $i$ and $j$ are $\sqsubseteq$-incomparable in $\calA$ and let us show that
  $a_i$ and $a_j$ are incomparable as sets. Let $\downarrow_\calA i = D_\alpha$ and $\downarrow_\calA j = D_\beta$.
  Since $X_1, \ldots, X_m$ are pairwise disjoint we have that
  $i \in D_\alpha$ and $i \notin D_\beta$ imply $X_\alpha \subseteq a_i$ and $X_\alpha \not\subseteq a_j$.
  Analogously,
  $j \notin D_\alpha$ and $j \in D_\beta$ imply $X_\beta \not\subseteq a_i$ and $X_\beta \subseteq a_j$.
  Therefore, $a_i$ and $a_j$ are incomparable.
  
  Finally, let us show that $i < j$ implies $a_i \clex a_j$. Assume that $i < j$ and $u = u_1 u_2 \ldots u_n$.
  If $i \sqsubseteq j$ then, as we have just seen, $a_i \supseteq a_j$, so $a_i \clex a_j$ because $\clex$ extends $\supseteq$.
  Assume, therefore, that $i \not\sqsubseteq j$. Then $i$ and $j$ are incomparable in $\calA$, whence follows that
  $a_i$ and $a_j$ are incomparable as sets (previous paragraph). Seeking a contradiction, assume that $a_j \clex a_i$.
  Then $s := \min(a_j \setminus a_i) < \min(a_i \setminus a_j)$. Since $s \in a_j$ then there is a $q$ such that
  $s \in X_q$ and $j \in D_q$. Note that $s \in X_q$ means that $u_s = x_q$. Let $\downarrow_\calA i = D_p$.
  Clearly we have that $\max(D_p) = i$ and $X_p \subseteq a_i$. From $\max(D_p) = i < j \in D_q$
  we easily conclude that $D_p \alex D_q$, whence $p < q$. Therefore, $t := \min(u^{-1}(x_p)) < \min(u^{-1}(x_q)) \le s$
  (because $u_s = x_q$). From $u_t = x_p$ it follows that $t \in X_p \subseteq a_i$.
  Next, we note that $t \in a_j$ (if $t \notin a_j$ then $\min(a_i \setminus a_j) \le t$ so
  $s = \min(a_j \setminus a_i) < \min(a_i \setminus a_j) \le t$ contradicts $t < s$).
  Since $X_1, \ldots, X_m$ are pairwise disjoint it follows that $t \in X_p \subseteq a_j$, whence
  $j \in D_p = \downarrow_\calA i$, so $j \sqsubseteq i$. Contradiction.
  
  So, the definition of $\Phi$ is correct. We still have to show that this family of maps
  satisfies the requirement~(PA) of Definition~\ref{opos.def.PA}.
  
  Let $\calB = (\{1, 2, \ldots, \ell\}, \Boxed\sqsubseteq, \Boxed<)$ be a linearly ordered poset that embeds into $\calA$.
  Let $D'_1$, \ldots, $D'_d$ be all the nonempty downsets in $\calB$ and let $D'_1 \alex D'_2 \alex \ldots \alex D'_d$.
  Take any embedding $f : \calB \hookrightarrow \calA$ and let us show that there is a word
  $h = h_1 h_2 \ldots h_m \in W^m_d(\{0\})$ such that $\Phi_{\calA, n}(u) \circ f = \Phi_{\calB, n}(u \cdot h)$.

  Define $h = h_1 h_2 \ldots h_m \in W^m_d(\{0\})$ as follows:
  $$
    h_i = \begin{cases}
      x_j, & f^{-1}(D_i) = D'_j\\
      0,   & \text{otherwise.}
    \end{cases}
  $$
  Let us first show that $h$ is indeed a $d$-parameter word. Since every downset in $\calB$ is an inverse image of a downset in $\calA$
  (for every $j$, $D'_j = f^{-1}(D_i)$ where $D_i = \downarrow_\calA f(D'_j)$), each of the variables $x_1, \ldots, x_d$ appears at
  least once in $h$. Let us show that $\min(h^{-1}(x_\alpha)) < \min(h^{-1}(x_\beta))$ whenever $1 \le \alpha < \beta \le d$.
  Take $\alpha$, $\beta$ such that $1 \le \alpha < \beta \le d$ and let $\min(h^{-1}(x_\beta)) = q$. Since $h_q = x_\beta$
  we know that $f^{-1}(D_q) = D'_\beta$. Take $p$ so that $D'_\alpha = f^{-1}(D_p)$. Then $h_p = x_\alpha$.
  If $p < q$ then $\min(h^{-1}(x_\alpha)) \le p < q = \min(h^{-1}(x_\beta))$ and we are done. Assume, therefore, that $p > q$.
  So, we have that $D_p \galex D_q$ (because $p > q$) and $f^{-1}(D_p) \alex f^{-1}(D_q)$ (because
  $\alpha < \beta$ whence $f^{-1}(D_p) = D'_\alpha \alex D'_\beta = f^{-1}(D_q)$).
  
  \medskip
  
  Claim. There exists a $D_r$ such that $D_r \alex D_q$ and $f^{-1}(D_r) = f^{-1}(D_p)$.
  
  Proof. Let $D_r = \downarrow_\calA f(f^{-1}(D_p))$. It is easy to show that $f^{-1}(D_r) = f^{-1}(D_p)$. Let us show that
  $D_r \alex D_q$. We know that $f^{-1}(D_p) \alex f^{-1}(D_q)$, so $f(f^{-1}(D_p)) \alex f(f^{-1}(D_q))$.
  Since $f(f^{-1}(D_q)) \subseteq D_q$ it follows that $f(f^{-1}(D_q)) \alex D_q$ (as $\alex$ extends $\subseteq$).
  Therefore, $f(f^{-1}(D_p)) \alex D_q$. Finally,
  $D_r = \Boxed{\downarrow_\calA} f(f^{-1}(D_p)) \alex \Boxed{\downarrow_\calA} D_q = D_q$ (because $D_q$ is a downset).
  This concludes the proof of the claim.
  
  \medskip
  
  Now, $D_r \alex D_q$ implies that $r < q$, while $f^{-1}(D_r) = f^{-1}(D_p) = D'_\alpha$ means that $h_r = x_\alpha$.
  Therefore, $\min(h^{-1}(x_\alpha)) \le r < q = \min(h^{-1}(x_\beta))$, which completes the proof that $h$ is a $d$-parameter word.
  
  Let $X'_i = (u \cdot h)^{-1}(x_i)$, $1 \le i \le d$, and $a'_j = \UNION \{X'_\beta : j \in D'_\beta\}$, $1 \le j \le |B|$.
  The following is a straightforward but useful observation:
  if $h^{-1}(x_\beta) = \{\alpha_1, \ldots, \alpha_s\}$ then
  $D'_\beta = f^{-1}(D_{\alpha_1}) = \ldots = f^{-1}(D_{\alpha_s})$
  and $X'_\beta = (u \cdot h)^{-1}(x_\beta) = u^{-1}(x_{\alpha_1}) \union \ldots \union u^{-1}(x_{\alpha_s})
  = X_{\alpha_1} \union \ldots \union X_{\alpha_s}$.
  
  In order to complete the proof it suffices to show that $a'_j = a_{f(j)}$ for all $1 \le j \le |B|$.
  
  $(\subseteq)$:
  Take any $X'_\beta \subseteq a'_j$. Then $j \in D'_\beta$. Let $h^{-1}(x_\beta) = \{\alpha_1, \ldots, \alpha_s\}$. Then
  $D'_\beta = f^{-1}(D_{\alpha_1}) = \ldots = f^{-1}(D_{\alpha_s})$, whence $j \in f^{-1}(D_{\alpha_i})$ for all
  $1 \le i \le s$. Consequently, $f(j) \in D_{\alpha_i}$ for all $1 \le i \le s$, so
  $X_{\alpha_i} \subseteq a_{f(j)}$ for all $1 \le i \le s$. Finally, $X'_\beta = X_{\alpha_1} \union \ldots \union X_{\alpha_s}
  \subseteq a_{f(j)}$.
  
  $(\supseteq)$:
  Take any $X_\alpha \subseteq a_{f(j)}$. Then $f(j) \in D_\alpha$ whence $j \in f^{-1}(D_\alpha) = D'_\beta$
  so $X'_\beta \subseteq a'_j$. By the definition of $h$ we have that $h_\alpha = x_\beta$ whence $X_\alpha \subseteq X'_\beta$.
  Therefore, $X_\alpha \subseteq a'_j$.
\end{proof}

\begin{THM} \cite{van-the-metric-spaces}
  The category $\catU_S$ has the Ramsey property for every set $S \subseteq \RR$ of nonnegative reals.
  Consequently, the category $\catU$ has the Ramsey property.
\end{THM}
\begin{proof}
  In order to show that $\catU_S$ has the Ramsey property for every set $S \subseteq \RR$ of nonnegative reals
  it suffices to show that the category $\catU_{S'}$ has the Ramsey property for every finite set $S' \subseteq \RR$
  of nonnegative reals. Namely, assume that $\catU_{S'}$ has the Ramsey property for every finite $S'$.
  Take any set $S \subseteq \RR$ of nonnegative reals,
  any $\calU, \calV \in \Ob(\catU_S)$ such that $\calU \hookrightarrow \calV$, and any $k \ge 2$.
  Since $\calV$ is finite, $S' = \spec(\calV)$ is a finite subset of $S$ and $\calU, \calV \in \Ob(\catU_{S'})$ because
  $\spec(\calU) \subseteq \spec(\calV)$.
  The category $\catU_{S'}$ has the Ramsey property by the assumption, so
  there is a $\calW \in \Ob(\catU_{S'})$ such that $\calW \longrightarrow (\calV)^\calU_k$.
  But $\catU_{S'}$ is a subcategory of $\catU_S$, whence $\calW \in \Ob(\catU_S)$.

  Now, take any \emph{finite} set $S \subseteq \RR$ of nonnegative reals.
  In order to show that $\catU(S)$ has the Ramsey property
  it suffices to establish a pre-adjunction
  $$
    F : \Ob(\catU_S) \rightleftarrows \Ob(\catP) : G
  $$
  since $\catP$ has the Ramsey property (Theorem~\ref{opos.thm.main}).

  Let us construct one such pre-adjunction. Let $S = \{0 = s_0 < s_1 < \ldots < s_k \} \subseteq \RR$.
  For $\calU = (U, d, \Boxed<) \in \Ob(\catU_S)$ put $B_\calU = \{B(x, s) : x \in U, s \in S\}$
  where $B(x, s) = \{y \in U : d(x, y) \le s\}$. Let us order the balls in $B_\calU$ as follows:
  $$
    B(x, s_i) \prec B(y, s_j) \text{ if and only if } s_i < s_j, \text{ or } s_i = s_j \text{ and } x < y.
  $$
  Note that $\prec$ is a linear ordering of $B_\calU$ because in an ultrametric space every point in a ball can serve as the
  center of the ball, and because $\calU$ is a convexly ordered ultrametric space.
  Now let
  $$
    F(\calU) = (B_\calU, \Boxed\subseteq, \Boxed\prec).
  $$
  Clearly, $(B_\calU, \Boxed\subseteq)$ is a poset, and we have just seen that $\prec$ is a linear ordering of $B_\calU$.
  It is easy to see that $\prec$ extends $\subseteq$: if $B(x, s_i) \subset B(y, s_j)$ then $s_i < s_j$ whence
  $B(x, s_i) \prec B(y, s_j)$. Therefore, $(B_\calU, \Boxed\subseteq, \Boxed\prec)$ is a finite linearly ordered poset and the
  definition of $F$ is correct.

  For $\calA = (A, \Boxed\sqsubseteq, \Boxed\prec) \in \Ob(\catP)$ put
  $$
    A^{<k} = \{(a_0, a_1, \ldots, a_{k-1}) : a_i \in A, 0 \le i \le k - 1 \}.
  $$
  (By this notation we want to stress that the tuples in $A^{<k}$ are indexed by $0, 1, \ldots, k - 1$.)
  Tuples $(a_0, a_1, \ldots, a_{k-1}), (b_0, b_1, \ldots, b_{k-1}), \ldots \in A^{<k}$ will be abbreviated as
  $\overline a$, $\overline b$, \ldots, respectively.
  Define $d_\calA : (A^{<k})^2 \to S$ by $d_\calA(\overline a, \overline b) = s_j$ where
  $$
    j = \min\{p \in \{0, 1, \ldots, k-1\} : (\forall i \ge p) a_i = b_i \},
  $$
  and we assume that $\min \0 = k$. Next, put $\overline a \precalex \overline b$ if and only if there is a $j$ such that
  $a_j \prec b_j$ and $(\forall i > j)a_i = b_i$. Finally, let
  $$
    G(\calA) = (A^{<k}, d_\calA, \precalex).
  $$
  
  \medskip
  
  Claim. $(A^{<k}, d_\calA, \precalex)$ is a convexly ordered ultrametric space with distances in $S$.
  
  Proof. It is clear that $d_\calA(\overline a, \overline b) \in S$ for all
  $\overline a, \overline b \in A^{<k}$ and that $\precalex$ is a linear ordering of $A^{<k}$.
  In order to show that $d_\calA$ is an ultrametric we will just demonstrate the triangle inequality, as the other axioms are
  obvious. Take any $\overline a, \overline b, \overline c \in A^{<k}$
  and let us show that $d_\calA(\overline a, \overline c) \le \max\{d_\calA(\overline a, \overline b), d_\calA(\overline b, \overline c)\}$.
  If $d_\calA(\overline a, \overline b) = s_k$ or $d_\calA(\overline b, \overline c) = s_k$ the inequality is trivially true. Assume,
  therefore, that $d_\calA(\overline a, \overline b) = s_i < s_k$ and $d_\calA(\overline b, \overline c) = s_j < s_k$.
  If $i \ge j$ then $a_\ell = b_\ell = c_\ell$ for $\ell \ge i$ whence
  $d_\calA(\overline a, \overline c) \le s_i = \max\{s_i, s_j\}$. The case $i < j$ is analogous.

  It still remains to be shown that every ball in $(A^{<k}, d_\calA)$ is convex with respect to $\precalex$. Let $\beta$ be a ball
  in $(A^{<k}, d_\calA)$. Take any $\overline a, \overline b \in \beta$ and a $\overline c \in A^{<k}$
  such that $\overline a \precalex \overline c \precalex \overline b$. Let $d_\calA(\overline a, \overline b) = s_i$.
  
  If $s_i = s_k$ then $A^{<k} = B(\overline a, s_i) = \beta$ so $\overline c \in \beta$ trivially. Assume, therefore,
  that $s_i < s_k$. From $\overline a \precalex \overline c$ it follows that there is a $t$ such that $a_t \prec c_t$
  and $a_j = c_j$ for all $j > t$. Let us show that $t < i$. Suppose, to the contrary, that $t \ge i$.
  Then $a_j = c_j = b_j$ for $j > t$. Since $a_t \prec c_t$ and $a_t = b_t$, it follows that
  $b_t \prec c_t$. Therefore, $\overline b \precalex \overline c$. Contradiction. So, $t < i$.
  
  Now, from $t < i$ it follows that $a_j = c_j$ for $j \ge i$, whence $d_\calA(\overline a, \overline c) \le s_i$.
  Therefore, $\overline c \in B(\overline a, s_i) \subseteq \beta$.
  This completes the proof of the claim.
  
  \medskip
  
  For $\calU = (U, d, \Boxed<) \in \Ob(\catU_S)$ and $\calA = (A, \Boxed\sqsubseteq, \Boxed\prec) \in \Ob(\catP)$
  let us define
  $$
    \Phi_{\calU, \calA} : \hom_{\catP}(F(\calU), \calA) \to \hom_{\catU_S}(\calU, G(\calA))
  $$
  as follows. For $u : F(\calU) \hookrightarrow \calA$ let $\hat u = \Phi_{\calU, \calA}(u) : U \to A^{<k}$ be defined by
  $$
    \hat u (x) = (u(B(x, s_0)), u(B(x, s_1)), \ldots, u(B(x, s_{k - 1}))).
  $$
  To show that the definition of $\Phi$ is correct we have to show that for every $u : F(\calU) \hookrightarrow \calA$ 
  the mapping $\hat u$ is an embedding $\calU \hookrightarrow G(\calA)$.

  To start with, note that $\hat u$ is injective: if $\hat u(x) = \hat u(y)$ then $u(B(x, s_0)) = u(B(y, s_0))$, whence
  $x = y$ having in mind that $u$ is injective and that $s_0 = 0$.

  It easy to show that $d(x, y) = d_\calA(\hat u(x), \hat u(y))$. Let $d(x, y) = s_i$. 
  Then $B(x, s_j) = B(y, s_j)$ for all $j \ge i$ and $B(x, s_{i-1}) \ne B(y, s_{i-1})$. Therefore,
  $d_\calA(\hat u(x), \hat u(y)) = s_i$ by definition.
  
  Finally, let us show that $x < y$ implies $\hat u(x) \precalex \hat u(y)$. Since $x \ne y$ we have that
  $\hat u(x) \ne \hat u(y)$, so there is an $i$ such that $u(B(x, s_i)) \ne u(B(y, s_i))$, or, equivalently,
  $B(x, s_i) \ne B(y, s_i)$. Assume that $i$ is the largest such index so that $u(B(x, s_j)) = u(B(y, s_j))$ for all $j > i$.
  Since $\calU$ is an ultrametric space it follows that $B(x, s_i) \sec B(y, s_i) = \0$, whence
  $B(x, s_i) \prec B(y, s_i)$ by definition of $\prec$. But then $u(B(x, s_i)) \prec u(B(y, s_i))$ because $u$ is an
  embedding. This, together with $u(B(x, s_j)) = u(B(y, s_j))$ for all $j > i$ yields $\hat u(x) \precalex \hat u(y)$.
  
  So, the definition of $\Phi$ is correct. We still have to show that this family of maps
  satisfies the requirement~(PA) of Definition~\ref{opos.def.PA}.

  Let $\calU' = (U', d', \Boxed<)$ be a linearly ordered ultrametric space that embeds into $\calU$ and let $f : \calU' \hookrightarrow \calU$
  be an embedding. Define $v : B_{\calU'} \to B_\calU$ by
  $$
    v(B(x, s_i)) = B(f(x), s_i).
  $$
  Because $\calU$ and $\calU'$ are ultrametric spaces and because $f$ is an embedding it follows immediately that
  $v$ does not depend on the choice of the center of the ball and that it is injective.
  
  Let us show that $v$ is an embedding $F(\calU') \hookrightarrow F(\calU)$.
  Assume, first, that $B(x, s_i) \subseteq B(y, s_j)$. Then $s_i \le s_j$ and $d'(x, y) \le s_j$, whence follows
  immediately that $B(f(x), s_i) \subseteq B(f(y), s_j)$. (If $z \in B(f(x), s_i)$ then $d(z, f(x)) \le s_i$, so
  $d(z, f(y)) \le \max\{d(z, f(x)), d(f(x), f(y))\} \le s_j$.) Conversely, if
  $B(f(x), s_i) \subseteq B(f(y), s_j)$ then $s_i \le s_j$ and $d(f(x), f(y)) \le s_j$, whence follows
  that $B(x, s_i) \subseteq B(y, s_j)$ because $f$ is an embedding.
  
  Assume now that $B(x, s_i) \prec B(y, s_j)$. Then $s_i < s_j$, or $s_i = s_j$ and $x < y$. If $s_i < s_j$ then
  $B(f(x), s_i) \prec B(f(y), s_j)$ by definition of~$\prec$. If, however, $s_i = s_j$ and $x < y$ then
  $f(x) < f(y)$ and again $B(f(x), s_i) \prec B(f(y), s_j)$ by definition of~$\prec$.
  
  Finally, let us show that $\Phi_{\calU, \calA}(u) \circ f = \Phi_{\calU', \calA}(u \circ v)$. Put $\hat u = \Phi_{\calU, \calA}(u)$
  and $\widehat{u \circ v} = \Phi_{\calU', \calA}(u \circ v)$. Then
  \begin{align*}
    \widehat{u \circ v}(x)
      &= \big(  u \circ v(B(x, s_0)), u \circ v(B(x, s_1)), \ldots, u \circ v(B(x, s_{k - 1}))  \big)\\
      &= \big(  u(B(f(x), s_0)), u(B(f(x), s_1)), \ldots, u(B(f(x), s_{k - 1}))  \big)\\
      &= \hat u (f(x)) = \hat u \circ f (x).
  \end{align*}
  This completes the proof.
\end{proof}

As the final demonstration of this strategy we shall show that the class of all finite linearly ordered metric spaces
has the Ramsey property.

Let $T = \{0 = t_0 < t_1 < \ldots < t_\ell \} \subseteq \RR$ be a finite set of nonnegative reals. We say that $T$ is \emph{tight}
if $t_{i + j} \le t_i + t_j$ for all $0 \le i \le j \le i + j \le \ell$.

\begin{LEM}\label{opos.lem.tight}
  Let $(A, \Boxed+)$ be a subsemigroup of the additive semigroup $(\RR, \Boxed+)$ such that $0 \in A \ne \{0\}$.
  For every finite set $S = \{0 = s_0 < s_1 < \ldots < s_k \} \subseteq A$ there exists a finite tight set
  $T = \{0 = t_0 < t_1 < \ldots < t_\ell \} \subseteq A$ such that $S \subseteq T$, $t_1 = s_1$ and $t_\ell = s_k$.
\end{LEM}
\begin{proof}
  Let $S = \{0 = s_0 < s_1 < \ldots < s_k \} \subseteq A$ be a finite set.
  Let us construct a sequence $t_0, t_1, t_2, \ldots$ of reals as follows. Let
  $t_0 = s_0 = 0$ and $t_1 = s_1$.
  Assume that we have constructed $t_0, t_1, \ldots, t_i$ and let
  $$
    \{t_0, t_1, \ldots, t_i\} \sec S = \{s_0, s_1, \ldots, s_j\}.
  $$
  If $j = k$ we stop with the construction. If, however, $j < k$ let
  $$
    m_{i+1} = \min\{ t_\alpha + t_\beta : \alpha + \beta = i + 1, 1 \le \alpha \le \beta \}
  $$
  and
  $$
    t_{i+1} = \begin{cases}
      s_{j+1}, & s_{j+1} \le m_{i + 1},\\
      m_{i + 1}, & s_{j + 1} > m_{i + 1}.
    \end{cases}
  $$

  Clearly, $\{t_0, t_1, t_2, \ldots\} \subseteq A$,
  $t_1 = s_1$, and $t_i \le m_i \le t_\alpha + t_\beta$ whenever $\alpha + \beta = i$, $1 \le \alpha \le \beta$.

  \medskip

  Next, let us show that $t_{i+1} > t_i$ for all $i \ge 0$. We proceed by induction. The first step $t_1 > t_0$ is obvious.
  Assume that $t_j > t_{j - 1}$ for all $j \le i$.

  Case 1: $t_{i + 1} = m_{i + 1}$. Take any $\alpha$ and  $\beta$ such that $\alpha + \beta = i + 1$ and $1 \le \alpha \le \beta$.
  Since $\beta \le i$, by the induction hypothesis we have that $t_\beta > t_{\beta - 1}$ whence
  $t_\alpha + t_\beta > t_\alpha + t_{\beta - 1} \ge t_{\alpha + \beta - 1} = t_i$. Therefore,
  $$
    m_{i+1} = \min\{ t_\alpha + t_\beta : \alpha + \beta = i + 1, 1 \le \alpha \le \beta \} > t_i.
  $$

  Case 2: $t_{i + 1} = s_{j + 1}$. If $t_i = s_j$ then $t_{i + 1} > t_i$ because $s_{j + 1} > s_j$. Assume, therefore,
  that $t_i \ne s_j$. Then $t_i = m_i < s_{j + 1}$ by construction.

  Therefore, $t_{i+1} > t_i$ for all $i \ge 0$.

  \medskip

  Finally, let us show that the procedure stops, whence follows that $t_\ell = s_k$ and that $S \subseteq T$.

  By construction $t_1 = s_1$. Assume that $s_j = t_i$ for some $j$ and $i$ and let us show that
  $s_{j + 1} = t_{n}$ for some $n > i$. Seeking a contradiction, suppose this is not the case.
  Then, by construction, $t_{i + \beta} = m_{i + \beta} < s_{j + 1}$ for all $\beta \in \NN$. Let
  $$
    \delta = \min \{t_\alpha - t_{\alpha - 1} : 1 \le \alpha \le i \}.
  $$
  Note that $\delta > 0$ because, as we have just seen, $t_\alpha > t_{\alpha - 1}$ for all $\alpha \ge 1$.
  Let us show that $t_{i + 1} - t_i \ge \delta$. Clearly,
  \begin{align*}
    t_{i + 1} - t_i &= \min \{ t_\alpha + t_{i + 1 - \alpha} : 1 \le \alpha \le i \} - t_i\\
                    &= \min \{ t_\alpha + t_{i + 1 - \alpha} - t_i : 1 \le \alpha \le i \}.
  \end{align*}
  Having in mind that $t_{i + 1 - \alpha} \ge t_{i - \alpha} + \delta$ and that $t_\alpha + t_{i - \alpha} \ge t_i$
  we obtain:
  \begin{align*}
    t_\alpha + t_{i + 1 - \alpha} - t_i  &\ge t_\alpha + t_{i - \alpha} + \delta - t_i \\
                                         &\ge t_i + \delta - t_i = \delta,
  \end{align*}
  whence
  $$
    t_{i + 1} - t_i = \min \{ t_\alpha + t_{i + 1 - \alpha} - t_i : 1 \le \alpha \le i \} \ge \delta.
  $$
  By the same argument we have that $t_{i + 2} - t_{i + 1} \ge \delta$, $t_{i + 3} - t_{i + 2} \ge \delta$, and so on.
  Therefore, $t_{i + \beta} \ge s_j + \beta \cdot \delta$ for all $\beta \in \NN$. This contradicts the assumption that
  $t_{i + \beta} < s_{j + 1}$ for all $\beta \in \NN$.
  
  Therefore, there is an $n > i$ such that $t_n = s_{j + 1}$. Eventually, the procedure stops with $t_\ell = s_k$.
\end{proof}

\begin{THM}\label{opos.thm.met}
  $(a)$ The category $\catM_S$ has the Ramsey property for every finite
  tight set $S = \{0 = s_0 < s_1 < \ldots < s_k \} \subseteq \RR$.

  $(b)$ Let $(A, \Boxed+)$ be a subsemigroup of the additive semigroup $(\RR, \Boxed+)$
  such that $0 \in A$ and let $I$ be an arbitrary interval of reals.
  Then $\catM_{\{0\} \union (I \sec A)}$ has the Ramsey property.

  $(c)$ \cite{Nesetril-metric} The categories $\catM$, $\catM_\QQ$ and $\catM_\ZZ$ have the Ramsey property.
\end{THM}
\begin{proof}
  $(a)$
  Let $S = \{0 = s_0 < s_1 < \ldots < s_k \} \subseteq \RR$ be a tight set.
  In order to show that $\catM_S$ has the Ramsey property it suffices to establish a pre-adjunction
  $$
    F : \Ob(\catM_S) \rightleftarrows \Ob(\catP) : G
  $$
  since $\catP$ has the Ramsey property (Theorem~\ref{opos.thm.main}).

  Take any $k \ge 1$. For $\calM = (M, d, \Boxed<) \in \Ob(\catM_S)$ put
  $$
    F(\calM) = (M \times \{0, 1, \ldots, k\}, \Boxed\sqsubseteq, \Boxed\prec),
  $$
  where
  $$
    (x, i) \sqsubseteq (y, j) \text{ if and only if } i \le j \text{ and } d(x, y) \le s_j - s_i,
  $$
  and
  $$
    (x, i) \prec (y, j) \text{ if and only if } i < j, \text{ or } i = j \text{ and } x < y.
  $$
  It is easy to see that $(x, i) \sqsubseteq (y, j)$ implies that $B(x, s_i) \subseteq B(y, s_j)$, so that,
  loosely speaking, the poset $(M \times \{0, 1, \ldots, k\}, \Boxed\sqsubseteq)$ corresponds to the poset
  formed by the balls of $M$ under inclusion.

  It is obvious that $\sqsubseteq$ is reflexive and transitive, and it is antisymmetric because $(x, i) \sqsubseteq (y, j)$
  and $i = j$ imply $x = y$ (since $d(x, y) \le s_j - s_i = 0$). It is also easy to see that $\prec$ extends $\sqsubseteq$:
  if $(x, i) \sqsubset (y, j)$ then $i < j$ (because $(x, i) \sqsubseteq (y, j)$
  and $i = j$ imply $x = y$), whence $(x, i) \prec (y, j)$ by definition.
  Therefore, $(M \times \{0, 1, \ldots, k\}, \Boxed\sqsubseteq, \Boxed\prec)$ is a linearly ordered poset and the definition
  of $F$ is correct.
    
  As in the case of ultrametric spaces, for $\calA = (A, \Boxed\sqsubseteq, \Boxed\prec) \in \Ob(\catP)$ put
  $$
    A^{<k} = \{(a_0, a_1, \ldots, a_{k-1}) : a_i \in A, 0 \le i \le k - 1\}.
  $$
  Define $d_\calA : (A^{<k})^2 \to S$ as follows:
  $$
    d_\calA(\overline a, \overline b) = s_j
  $$
  where
  $$
    j = \min\{ p \in \{0, 1, \ldots, k-1\} : (\forall i \le k - 1 - p)(a_i \sqsubseteq b_{i + p} \land b_i \sqsubseteq a_{i + p}) \},
  $$
  and $\min \0 = k$. Next, put $\overline a \preclex \overline b$ if and only if there is a $j$ such that
  $a_j \prec b_j$ and $(\forall i < j)(a_i = b_i)$. Finally, let
  $$
    G(\calA) = (A^{<k}, d_\calA, \preclex).
  $$
  
  \medskip
  
  Claim. $(A^{<k}, d_\calA, \preclex)$ is a linearly ordered metric space with distances in $S$.
  
  Proof. It is clear that $d_\calA(\overline a, \overline b) \in S$ for all
  $\overline a, \overline b \in A^{<k}$ and that $\preclex$ is a linear ordering of $A^{<k}$. Let us show that $d_\calA$
  is a metric.
  
  Clearly, $d_\calA(\overline a, \overline a) = s_0 = 0$ and $d_\calA(\overline a, \overline b) = d_\calA(\overline b, \overline a)$
  for all $\overline a, \overline b \in A^{<k}$. Assume that $d_\calA(\overline a, \overline b) = 0$. Then
  $(\forall i \le k - 1)(a_i \sqsubseteq b_{i} \land b_i \sqsubseteq a_{i})$, whence $\overline a = \overline b$ because
  $\sqsubseteq$ is antisymmetric.
  
  Finally, let us show that $d_\calA(\overline a, \overline c) \le d_\calA(\overline a, \overline b) + d_\calA(\overline b, \overline c)$.
  Let $d_\calA(\overline a, \overline b) = s_p$ and $d_\calA(\overline b, \overline c) = s_q$. If $p + q \ge k$
  then $d_\calA(\overline a, \overline c) \le s_k \le s_p + s_{k - p} \le s_p + s_q$ because $S$ is tight and
  $k - p \le q \Rightarrow s_{k - p} \le s_q$.
  Assume, therefore, that $p + q \le k - 1$. Then
  $$
    (\forall i \le k - 1 - p)(a_i \sqsubseteq b_{i + p} \land b_i \sqsubseteq a_{i + p}), \text{ and}
  $$
  $$
    (\forall i \le k - 1 - q)(b_i \sqsubseteq c_{i + q} \land c_i \sqsubseteq b_{i + q}),
  $$
  whence
  $$
    (\forall i \le k - 1 - (p + q))(a_i \sqsubseteq c_{i + p + q} \land c_i \sqsubseteq a_{i + p + q}).
  $$
  Therefore, $d_\calA(\overline a, \overline c) \le s_{p + q} \le s_p + s_q$ because $S$ is tight. This completes the proof of the claim.
  
  \medskip
  
  For $\calM = (M, d, \Boxed<) \in \Ob(\catM_S)$ and $\calA = (A, \Boxed\sqsubseteq, \Boxed\prec) \in \Ob(\catP)$
  let us define
  $$
    \Phi_{\calM, \calA} : \hom_{\catP}(F(\calM), \calA) \to \hom_{\catM_S}(\calM, G(\calA))
  $$
  as follows. For $u : F(\calM) \hookrightarrow \calA$ let $\hat u = \Phi_{\calM, \calA}(u) : M \to A^{<k}$ be defined by
  $$
    \hat u (x) = (u(x, 0), u(x, 1), \ldots, u(x, k - 1)).
  $$
  To show that the definition of $\Phi$ is correct we have to show that for every $u : F(\calM) \hookrightarrow \calA$ 
  the mapping $\hat u$ is an embedding $\calM \hookrightarrow G(\calA)$.
  
  To start with, note that $x < y$ implies $\hat u(x) \preclex \hat u(y)$ straightforwardly:
  $x < y$ implies $(x, 0) \prec (y, 0)$ in $F(\calM)$, whence $u(x, 0) \prec u(y, 0)$ in $\calA$, so
  $\hat u(x) \preclex \hat u(y)$. Let us now show that $d(x, y) = d_\calA(\hat u(x), \hat u(y))$.
  
  Case 1: $d(x, y) = 0$. Then $\hat u(x) = \hat u(y)$, so $d_\calA(\hat u(x), \hat u(y)) = 0$.
  
  Case 2: $d(x, y) = s_k$. By construction, $d_\calA(\hat u(x), \hat u(y)) \le s_k$. Assume that
  $d_\calA(\hat u(x), \hat u(y)) = s_q < s_k$. Then $u(x, 0) \sqsubseteq u(y, q)$ in $\calA$.
  Since $u$ is an embedding, it follows that $(x, 0) \sqsubseteq (y, q)$ in $F(\calM)$, whence by definition
  $d(x, y) \le s_q < s_k$. Contradiction.

  Case 3: $d(x, y) = s_p$ where $1 \le p \le k - 1$. Then for all $i \le k - 1 - p$ we have that $(x, i) \sqsubseteq (y, i + p)$ and $(y, i) \sqsubseteq (x, i + p)$ in $F(\calM)$.
  Consequently, for all $i$ we have that $u(x, i) \sqsubseteq u(y, i + p)$ and $u(y, i) \sqsubseteq u(x, i + p)$ in $\calA$.
  Therefore, $d_\calA(\hat u(x), \hat u(y)) \le s_p$. If $d_\calA(\hat u(x), \hat u(y)) = s_q < s_p$ we reach a contradiction as in Case~2.
  
  So, the definition of $\Phi$ is correct. We still have to show that this family of maps
  satisfies the requirement~(PA) of Definition~\ref{opos.def.PA}.
  
  Let $\calM' = (M', d', \Boxed<)$ be a linearly ordered metric space that embeds into $\calM$ and let $f : \calM' \hookrightarrow \calM$
  be an embedding. Define $v : M' \times \{0, 1, \ldots, k\} \to M \times \{0, 1, \ldots, k\}$ by
  $$
    v(x, i) = (f(x), i).
  $$
  Let us show that $v$ is an embedding $F(\calM') \hookrightarrow F(\calM)$. Clearly,
  $(x, i) \sqsubseteq (y, j)$ in $F(\calM')$ iff $i \le j$ and $d'(x, y) \le s_j - s_i$
  iff $i \le j$ and $d(f(x), f(y)) \le s_j - s_i$ iff $(f(x), i) \sqsubseteq (f(y), j)$ in $F(\calM)$. Assume now that
  $(x, i) \prec (y, j)$ in $F(\calM')$. If $i < j$ then $(f(x), i) \prec (f(y), j)$ in $F(\calM)$. If, however, $i = j$
  then $x < y$, so $f(x) < f(y)$ and again $(f(x), i) \prec (f(y), j)$ in $F(\calM)$.
  
  Finally, let us show that $\Phi_{\calM, \calA}(u) \circ f = \Phi_{\calM', \calA}(u \circ v)$. Put $\hat u = \Phi_{\calM, \calA}(u)$
  and $\widehat{u \circ v} = \Phi_{\calM', \calA}(u \circ v)$. Then
  \begin{align*}
    \widehat{u \circ v}(x)
      &= \big(  u \circ v(x, 0), u \circ v(x, 1), \ldots, u \circ v(x, k - 1)  \big)\\
      &= \big(  u(f(x), 0), u(f(x), 1), \ldots, u(f(x), k - 1)  \big)\\
      &= \hat u (f(x)) = \hat u \circ f (x).
  \end{align*}

  $(b)$ If $\{0\} \union (I \sec A) = \{0\}$ the statement is trivially true. Assume, therefore, that
  $\{0\} \union (I \sec A) \ne \{0\}$. Then $A \ne \{0\}$.
  Take any $\calU, \calV \in \Ob(\catM_{\{0\} \union (I \sec A)})$ such that $\calU \hookrightarrow \calV$, and any $k \ge 2$.
  Since $\calV$ is finite, $S = \spec(\calV)$ is a finite subset of~$A$. By Lemma~\ref{opos.lem.tight}
  there is a finite tight set $T \subseteq A$ which contains $S$. Then $\calU, \calV \in \Ob(\catM_T)$ because
  $\spec(\calU) \subseteq \spec(\calV) = S \subseteq T$. The category $\catM_T$ has the Ramsey property by $(a)$, so
  there is a $\calW \in \Ob(\catM_T)$ such that $\calW \longrightarrow (\calV)^\calU_k$.
  Since, by construction, the smallest and the largest nonzero
  elements of $S$ and $T$ coincide and since $S \subseteq \{0\} \union (I \sec A)$
  it follows that $T \subseteq \{0\} \union (I \sec A)$, so $\catM_T$ is a full subcategory of $\catM_{\{0\} \union (I \sec A)}$
  whence $\calW \in \Ob(\catM_{\{0\} \union (I \sec A)})$.

  $(c)$ Directly from $(b)$.
\end{proof}

As a closing remark let us mention briefly the history of proving the Theorem~\ref{opos.thm.met}.
This result was proven in \cite{Nesetril-metric} and independently in \cite{DR12}. Ramsey
expansions of $S$-metric spaces for sets $S$ consisting of at most 4 elements are
shown in \cite{van-the-memoirs}. Finally \cite{Nesetril-AllThose} shows the
existence of a Ramsey expansion for all meaningful choices of $S$.
Interestingly, all these proofs are based on the partite construction.
In contrast to that, our proof of Theorem~\ref{opos.thm.met} follows
directly from the Graham-Rothschild Theorem.

\section{Acknowledgements}

The author would like to thank Miodrag Soki\'c as well as two anonimous referees for many valuable comments
which have significantly improved the presentation of the results and the historical background of the problems
treated in this paper.

The author gratefully acknowledges the support of the Grant No.~174019
of the Ministry of Education, Science and Technological Development of the
Republic of Serbia.

\end{document}